\documentclass[12pt]{amsart}
\usepackage{amsmath}
\usepackage{amsthm}
\usepackage{amssymb}
\usepackage{amscd}
\usepackage{amsfonts}
\usepackage{amsbsy}
\usepackage{graphicx}
\usepackage[dvips]{psfrag}
\usepackage{array}
\usepackage{hyperref}

\newtheorem {theorem} {Theorem}

\newtheorem {corollary} [theorem]{Corollary}

\newcommand{\R}{\mathbb{R}}
\newcommand{\z}{z}
\newcommand{\sign}{\operatorname{sign}}
\newcommand{\der}{~\mathrm{d}}
\def\dis{\displaystyle}
\def\l{\ell}
\def\e{\varepsilon}

\begin{document}

\title[Birth of limit cycles]
{Birth of limit cycles for a class of continuous and discontinuous
differential systems in $(d+2)$--dimension}

\author[J. Llibre, M.A. Teixeira and I.O. Zeli ]
{Jaume Llibre$^1$, Marco A. Teixeira$^2$ and Iris O. Zeli$^{2}$}

\address{$^1$ Departament de Matem\`{a}tiques,
Universitat Aut\`{o}noma de Barcelona, 08193 Bellaterra, Barcelona,
Catalonia, Spain.} \email{jllibre@mat.uab.cat}

\address{$^{2}$  Departamento de Matem\'{a}tica, Universidade
Estadual de Campinas, CP 6065, 13083-859, Campinas, SP, Brazil.}
\email{teixeira@ime.unicamp.br, irisfalkoliv@ime.unicamp.br}



\begin{abstract}
The orbits of the reversible differential system $\dot{x}=-y$, $\dot{y}=x$,
$\dot{\z}=0$, with $x,y \in \R$ and $\z\in \R^d$, are periodic with
the exception of the equilibrium points $(0,0, \z)$.
We compute the maximum number of limit cycles which bifurcate from
the periodic orbits of the system $\dot{x}=-y$, $\dot{y}=x$,
$\dot{\z}=0$, using the averaging theory of first order, when this
system is perturbed, first inside the class of all polynomial
differential systems of degree $n$, and second inside the class of
all discontinuous piecewise polynomial differential systems of
degree $n$ with two pieces, one in $y> 0$ and the other in $y<0$. In
the first case this maximum number is $n^d(n-1)/2$, and in the
second is $n^{d+1}$.
\end{abstract}

\maketitle

\section{Introduction and statements of the main results}

Limit cycles have been used to model the behavior of many real
process and different modern devices.  In general to prove the
existence of limit cycles is a very difficult problem. One way to
produce limit cycles is perturbing  differential systems that have a
linear center. In this case, the limit cycles in a perturbed system
bifurcate from the periodic orbits of the unperturbed center.  The
search  for the maximum number of limit cycles that polynomial
differential systems of a given degree can have  is part of {\it
$16^{th}$ Hilbert's Problem} and many contributions have been made
in this direction, see for instance \cite{Hil,Ily,Lii} and the
references quoted therein.

\smallskip

Recently the theory of limit cycles has also been studied in
discontinuous piecewise differential systems. The analysis of these
systems can be traced from Andronov {\it et al.} \cite{AVK} and
still continues to receive attention by researchers. Discontinuous
piecewise differential systems is a subject that have been developed
very fast due to its strong applications to  other branches of
science. Currently such systems are one of the connections between
mathematics, physics and engineering. These systems model several
phenomena in control systems, impact in mechanical systems,
nonlinear oscillations and  economics  see for instance
\cite{Bar,BSC,Bro,Chi,Ito,Min}. Recently they have been shown to be
also relevant as idealized models for biology \cite{Kri} and models
of cell activity \cite{Co,To,TG}. For more details see Teixeira
\cite{Tei} and all references therein.

\smallskip

As we have said it is not simple to determine the existence of limit
cycles in a  differential system. The simplest case for determining
limit cycles is in planar continuous piecewise linear systems when
they have only two linear differential systems separated by a straight line.
Even in this simple case, only after a delicate analysis it was
possible to show the existence of at most of  one limit cycle for
such systems, see \cite{FPRT} or an easier proof in \cite{LMP}.

\smallskip

Planar discontinuous piecewise linear differential systems with only
two linearity regions separated by a straight line have been studied
recently in \cite{HZ, HY}, among other papers. In  \cite{HZ} some
results about the existence of two limit cycles appeared, so that
the authors conjectured that the maximum number of limit cycles for
this class of piecewise linear differential systems is exactly two.
However in \cite{HY} strong numerical evidence about the existence
of three limit cycles was obtained.  As far we know the example in
\cite{HY} represents the first discontinuous piecewise linear
differential system with two zones with $3$ limit cycles surrounding
a unique equilibrium. Recently in \cite{LP} it is  proved that such
a system really has three limit cycles.

\smallskip

There are several papers studying the limit cycles of the continuous
piecewise linear differential systems in $\R^3$, see for instance
\cite{CLT,LP1,LPR,LPRT,LPT}. Our goal is study  the periodic solutions
of discontinuous piecewise polynomial differential systems in
$\R^{d+2}$.  More precisely the objective of this paper is to study
the existence of limit cycles in  continuous and discontinuous
piecewise polynomial differential systems in $\R^{d+2}$, where the
discontinuous differential system has two zones of continuity
separated by a hyperplane.  Without loss of generality we shall assume that the set of discontinuity is the hyperplane $y=0$
in $\R^{d+2}$. So we consider the linear differential  system in
$\R^{d+2}$ given by

\begin{align}\label{eq.system.linear}
 \dot x = &  - y \nonumber\\
 \dot y = & ~x \\
 \dot \z_\l = & ~0 \nonumber
\end{align}
 
 for $\l=1, \ldots, d$ and $x,y \in \R$, $\z \in \R^d$, where the dot denotes derivative with respect to the time $t$, which is reversible with respect to $\phi(x,y,\z)=(x,-y,\z)$ . First we are interested in studying the existence of limit cycles of the continuous polynomial differential system given by

\begin{align}
\label{eq.systemX.perturbed}
    \dot{x} = & -y + \e P_a(x,y,\z) ,\nonumber\\
    \dot{y} = & ~ x + \e P_b(x,y,\z),\\
    \dot{\z_\l} = & ~\e P_{c_\l}(x,y,\z), \nonumber
 \end{align}
 
 and after we also study the existence of limit cycles of the discontinuous piecewise polynomial differential system formed by two polynomial differential systems separated by the  hyperplane
 $y=0$,  namely
 
 \begin{equation}
  \label{eq.perturbed.discontinuous}
 \begin{array}{c}
 \left.
 \begin{array}{ll}
\dot{x} ~ =  & -y + \e P_a(x,y,\z) , \vspace*{0.15 cm}\\
\dot{y} ~ = &~ x + \e P_b(x,y,\z), \vspace*{0.15 cm}\\
\dot{\z_\l} = & ~\e P_{c_\l}(x,y,\z), 
 \end{array}
 \right\} \quad \hbox{if}\quad y>0 \vspace*{0.25 cm} \\ 
 \left.
 \begin{array}{ll}
 \dot{x} ~=  & -y + \e Q_a(x,y,\z) , \vspace*{0.15 cm}\\
 \dot{y} ~= &~ x + \e Q_b(x,y,\z), \vspace*{0.15 cm}\\
 \dot{\z_\l} = & ~\e Q_{c_\l}(x,y,\z),
 \end{array}
 \right\} \quad \hbox{if}\quad y < 0 
 \end{array}
\end{equation}
 
\noindent where $\e \neq 0$ is a small parameter and $\l=1, \ldots,d$. In this systems the polynomials $P_a$, $P_b$, $P_{c_\l}$, $Q_\alpha$, $Q_\beta$, $Q_{\gamma_\l}$ are   of degree $n$ in the variables $x$, $y$ and $\z$, more precisely
\begin{align*}
P_a(x,y,z) = &\sum_{i+j+k=0}^{n} a_{ijk} x^i y^j z^k, \quad
P_b(x,y,z)= \sum_{i+j+k=0}^{n} b_{ijk} x^i y^j z^k, \\
P_{c_\l}(x,y,z)= &\sum_{i+j+k=0}^{n} c_{\l ijk} x^i y^j z^k, \quad Q_a(x,y,z) = \sum_{i+j+k=0}^{n} \alpha_{ijk} x^i y^j z^k ,\\
Q_\beta(x,y,z)= & \sum_{i+j+k=0}^{n} \beta_{ijk} x^i y^j z^k, \quad Q_{\gamma_\l}(x,y,z)= \sum_{i+j+k=0}^{n} \gamma_{\l ijk} x^i y^j z^k.
\nonumber
\end{align*}

In this expressions  $k$ is a multi-index and  $i+j+k$ denotes
$i+j+k_1 + \ldots + k_d$, $\z^k$ denotes the product $z_1^{k_1} \ldots
z_d^{k_d}$ with $\z=(z_1,\ldots,z_d)\in \R^d$, and $a_{ijk}$ denotes
the coefficient $a_{ijk_1 \ldots k_d}$ of $x^i y^j z_1^{k_1} \ldots
z_d^{k_d}$.

\smallskip

It is clear that systems \eqref{eq.systemX.perturbed} and
\eqref{eq.perturbed.discontinuous} coincide for  $\varepsilon = 0$
and they have linear centers at every plane $\z=$ constant.  In this
paper we establish for $\varepsilon \neq 0$ sufficiently small the
maximum number of limit cycle of these systems  that bifurcate from
the periodic orbits of these linear centers using the averaging
theory of first order. The following result presents the results for
the continuous case.

\begin{theorem}\label{th.cycles.systemX}
Using  the averaging theory of first order for $|\varepsilon| \neq
0$ sufficiently small the maximum number of limit cycles of the
polynomial differential system \eqref{eq.systemX.perturbed} is at
most $n^d(n-1)/2$, and this number is reached.
\end{theorem}

In the next theorem we present results for the discontinuous
piecewise polynomial differential system
\eqref{eq.perturbed.discontinuous}.

\begin{theorem}\label{th.cycles.discontinuous}
Using the averaging theory of first order for $|\varepsilon| \neq 0$
sufficiently small the maximum number of limit cycles of the
discontinuous piecewise polynomial differential system
\eqref{eq.perturbed.discontinuous} is at most $n^{d+1}$, and this
number is  reached.
\end{theorem}

\begin{corollary}\label{cor.1}
Under the assumptions of Theorem \ref{th.cycles.systemX} if
additionally $a_{00k}=b_{00k}=0$ for all $k$, the limit cycles can
be chosen as close to the origin of $\R^{d+2}$ as we want.
\end{corollary}

\begin{corollary}\label{cor.2} 
Under the assumptions of Theorem \ref{th.cycles.discontinuous} if additionally  $a_{00k}=b_{00k}=\alpha_{00k}=\beta_{00k}=0$ for all $k$, the number of limit cycles of the discontinuous piecewise polynomial differential system \eqref{eq.perturbed.discontinuous} is at most $n^{d}(n-1)$, and this number is reached. Furthermore,  the limit cycles can be chosen as close to the origin of $\R^{d+2}$ as we want.
\end{corollary}

Corollaries \ref{cor.1} and \ref{cor.2} provides information of the
Hopf bifurcation of systems \eqref{eq.systemX.perturbed} and
\eqref{eq.perturbed.discontinuous}. More precisely, Corollaries
\ref{cor.1} and \ref{cor.2} show that at least $n^d(n-1)/2$ and
$n^d(n-1)$ limit cycles of systems \eqref{eq.systemX.perturbed} and
\eqref{eq.perturbed.discontinuous} can bifurcate from the origin of
$\R^{d+2}$, respectively. The results of Corollary \ref{cor.1}  in
the particular case $n=2$ coincides with the result obtained in
Theorem 1 of \cite{LZ}.

\smallskip

To prove these results we use the classical averaging theory, see
for instance \cite{SVM,Ve} for a general introduction to this
subject. This theory have been used for years to deal with
continuous differential systems.  The principle of averaging has
been extended in many directions and recently in \cite{LNT} the
authors extend the averaging theory for  detecting limit cycles of
certain discontinuous piecewise differential systems, via the
Brouwer degree and the regularization theory.

\smallskip

As far as we know this method is one of the best methods for
determining limit cycles in discontinuous piecewise differential
systems and has already been used by some authors. In \cite{LLM} the
method is used for determining the maximum number of limit cycles
that bifurcate from the periodic solutions of some family of
isochronous cubic polynomial centers perturbed by discontinuous
piecewise cubic polynomial differential systems with two zones
separated by a straight line.  In \cite{LM} limit cycles for
discontinuous piecewise quadratic differential systems with two
zones was studied using the averaging theory. Also in \cite{No}  the
averaging theory was applied to provide sufficient conditions for the
existence of limit cycles of discontinuous  perturbed planar centers
when the discontinuity set is a union of regular curves.

\smallskip

We have organized  this paper as follows. In Section
\ref{section:preliminaries} we briefly present notation and basic
concepts of the averaging theory of first order for continuous
differential systems (see Theorem \ref{Th.continuous.averaging})
and for discontinuous differential systems (see Theorem
\ref{Th.discontinuous.averaging}). In
Section~\ref{section:proofs} we present the proof of results of this paper. More precisely, the proof of Theorem \ref{th.cycles.systemX} is presented in the Subsection \ref{section:proof.th.continuous}, and in the
Subsection~\ref{section:proof.th.discontinuous} we prove the Theorem
\ref{th.cycles.discontinuous}. Finally  in the Subsection
\ref{section:proof.cor} we presented the proofs  of the Corollaries \ref{cor.1} and \ref{cor.2}.

\section{Basic results in averaging theory}
\label{section:preliminaries}

In this section we present the basic results from the averaging
theory of first order that we shall use for proving the results of
this paper. The following theorem provides a method for studying the
existence of periodic orbits of a differential system. For more
details on the averaging method see for instance \cite{Ve}.

\smallskip

Let $D$ be an open subset of $\R^n$. We denote the points of $\R
\times D$ as $(t,x)$, and we take the variable $t$ as the time.

\begin{theorem}\label{Th.continuous.averaging}
Consider the differential system
\begin{equation}
\label{eq.system.th.continuous}
\dot{x}=\varepsilon F(t,x) + \varepsilon^2 R(t,x,\varepsilon),\quad x(0)=0,
\end{equation}
where $F: \R \times D \to \R^n$ and $R: \R \times U \times
(-\varepsilon,\varepsilon) \to \R^n$ are con\-ti\-nuous functions and
$T-$periodic in the first variable. Define the averaging function
$f:D \to \R^n$ as
\begin{equation}
\label{eq.averaging.function.th}
f(x)= \int_0^T F(s,x) \der s,
\end{equation}
and assume that
\begin{itemize}
\item[(i)] the functions $F,~ R, ~D_xF,~ D_x^2F$ and $D_xR$ are defined,
continuous and bounded by a constant $M$ (independent of
$\varepsilon$) in $[0,\infty) \times D$ and for $\varepsilon \in
(0,\varepsilon_0]$,

\item[(ii)] for $p \in D$ with $f(p)=0$ we have $|J_f(p)| \neq 0$, where
$|J_f(p)|$ denotes the determinant of the Jacobian matrix of $f$
evaluated at $p$.
\end{itemize}
Then for $|\varepsilon| > 0$ sufficiently small there exists a
$T-$periodic solution $x(t,\varepsilon)$ of the system
\eqref{eq.system.th.continuous} such that $x(0,\varepsilon) \to p$
as $\varepsilon \to 0$.
\end{theorem}

Now let $h: \R \times D \to \R$ be a $\mathcal{C}^1$ function with
$0\in \R$ as a regular value, and $\Sigma = h^{-1}(0)$. Let $X,Y: \R
\times D \to \R^n$ be two continuous vector fields and assume that
$h,X$ and $Y$ are $T-$periodic in the variable $t$. We define a
\textit{discontinuous piecewise differential system} as

\begin{equation}
\label{eq.def.discontinuous.system}
\dot{x}=Z(t,x)=
\left\{
\begin{array}{c}
             X(t,x)\quad \text{if} \quad h(t,x) >0,\\
            Y(t,x) \quad \text{if} \quad h(t,x) <0.
\end{array}
\right.
\end{equation}
We rewrite the discontinuous differential system as follows.
Consider the sign function defined in $\R \backslash\left\{ 0
\right\}$ as
\begin{equation} \nonumber
\text{sign}(u) = \left\{
\begin{array}{c}
~~1 ~\quad \text{if} \quad u>0,\\
-1 \quad \text{if} \quad u<0.
\end{array}
\right.
\end{equation}
Then system \eqref{eq.def.discontinuous.system} can be written as
\begin{equation}\nonumber
\dot{x}=Z(t,x) = F_1(t,x) + \sign \left( h(t,x) \right) F_2(t,x),
\end{equation}

where
$$
F_1(t,x)= \dis\frac{1}{2} \left( X(t,x) + Y(t,x) \right)   \quad
\text{and} \quad   F_2(t,x)= \dis\frac{1}{2} \left( X(t,x) - Y(t,x)
\right).
$$

The following theorem is a version of Theorem
\ref{Th.continuous.averaging} for studying the periodic solutions of
discontinuous differential systems.

\begin{theorem} \label{Th.discontinuous.averaging}
Consider the discontinuous differential system
\begin{equation}
\label{eq.system1.th.discontinuous.averaging}
\dot{x}= \varepsilon F(t,x) + \varepsilon^2 R(t,x,\varepsilon),
\end{equation}

with

\begin{align*}
F(t,x)&= F_1(t,x) + \sign \left( h(t,x) \right) F_2(t,x),\\
R(t,x,\varepsilon)&= R_1(t,x,\varepsilon) + \sign \left( h(t,x)
\right) R_2(t,x,\varepsilon),
\end{align*}

where $F_1, F_2: \R \times D \to \R^n$, $R_1,R_2 : \R \times D
\times (-\varepsilon, \varepsilon) \to \R^n$ and $h: \R \times D \to
\R$ are continuous functions, $T-$periodic in the variable $t$. We
also suppose that $h$ is a $\mathcal{C}^1$ function with $0$ as a
regular value and we denote $\Sigma = h^{-1}(0)$. Define the
averaged function $f: D \to \R^n$ as

\begin{equation}
\label{eq.2.th.averaging.discontinuous}
f(x) = \int_0^T F(s,x) \der s,
\end{equation}
and assume that
\begin{itemize}
\item[(i)] the functions $F_1, F_2, R_1, R_2$ and $h$ are locally
Lipschitz with respect to $x$;

\item[(ii)] $\dis\frac{\partial h}{ \partial t}(t,x) \neq 0$ for
all $(t,x) \in
\Sigma$;

\item[(iii)] for $p \in C$ with $f(p)=0$, there exist a neighborhood
$U \subset C$ of $p$  such  that $f(z) \neq 0$ for all $ z \in
\overline{U} \backslash \left\{ p \right\}$ and $d_B(f,U,0) \neq 0$
($d_B$ is the Brouwer degree of $f$ in $p$).
\end{itemize}

Then for  $|\varepsilon| >0$ sufficiently small there exists a
$T-$periodic solutions $x(t,x)$ of system
\eqref{eq.system1.th.discontinuous.averaging} such that $ x(t,
\varepsilon) \to p$ as $\varepsilon \to 0$.
\end{theorem}

For a proof of Theorem \ref{Th.discontinuous.averaging} see Theorem
$A$ and Proposition $2$ in \cite{LNT}. Here we emphasize that if $f$
in \eqref{eq.2.th.averaging.discontinuous} is $C^1$ then the
hypotheses $d_B(f,U,0) \neq 0$ holds if $|J_f(p)| \neq 0$, see for
more details \cite{Ll}.

\section{Proof of the main results}
\label{section:proofs}
We devoted this section for the proof of results of this paper. For this we consider the notations introduced in the Introduction.

\subsection{Proof of Theorem \ref{th.cycles.systemX}}
\label{section:proof.th.continuous}

Applying the change of variables
\begin{equation}
\label{eq.polar.coordinates}
x = r \cos\theta, \quad y =r \sin \theta \quad \text{and} \quad \z = \z,
\end{equation}
system \eqref{eq.systemX.perturbed} becomes
\begin{equation} \nonumber
\label{1*}
\begin{array}{lll}
\dot{\theta}&=& 1+ \displaystyle
\frac{\varepsilon}{r}\sum_{i+j+k=0}^{n} r^{i+j} \z^k \big(a_{ijk}
\cos^i\theta \sin^{j+1}\theta  + b_{ijk} \cos^{i+1}\theta
\sin^j\theta \big),\\
\dot{r}&=&\displaystyle \varepsilon \sum_{i+j+k=0}^{n} r^{i+j} \z^k
\left( a_{ijk} \cos^{i+1}\theta \sin^j\theta  + 
b_{ijk}\cos^{i}\theta
\sin^{j+1}\theta\right),\\
\dot{z}_{l}&= &\displaystyle \varepsilon\sum_{i+j+k=0}^{n}
c_{lijk}~r^{i+j} \z^k \cos^i\theta \sin^j\theta,
\end{array}
\end{equation} 

Essentially for $\varepsilon\ne 0$ sufficiently small  we study the existence of limit cycles  bifurcating  from  periodic
orbits of this system when $\varepsilon=0$, and that are contained in the
cylindrical annulus
\begin{equation}\label{2*}
\tilde{A}= \lbrace (\theta,r,z): r_0 \leq r \leq r_1,~ \theta\in
\mathbb S^1, ~ \z\in \R^d \rbrace.
\end{equation}
So for $\varepsilon$ sufficiently small $\dot{\theta}
>0$ for every $(\theta,r,\z) \in \tilde{A}$.

\smallskip

Now taken as new independent variable $\theta$ instead of $t$ so
\eqref{eq.systemX.perturbed} in $\tilde{A}$ can be written as
\begin{equation}
\label{eq.systemX.new.coordinates}
\begin{array}{lll}
r' &=& ~\e F_1(\theta,r, \z) +  \mathcal{O}(\e^2) \vspace*{0.15cm}\\
z'_\l &= & ~ \e  F_{\l+1}(\theta,r,z) + \mathcal{O}(\e^2) ,
\end{array}
\end{equation}

for $\l=1, \ldots,d$,  where the prime denotes derivative with respect to the variable
$\theta$, and
\begin{equation}	 
\begin{array}{lll}
\label{eq.F1.F2}
F_1(\theta,r,z) &= &\dis \sum_{i+j+k=0}^{n} r^{i+j}z^k \big( a_{ijk}
\cos^{i+1}\theta \sin^j\theta \vspace*{0.15cm}\\
 & & + b_{ijk} \cos^{i}\theta 
\sin^{j+1}\theta \big), \vspace*{0.15cm}\\
F_{\l+1}(\theta,r,z) &=& \dis \sum_{i+j+k=0}^{n} r^{i+j} z^k~c_{\l ijk}
\cos^i\theta \sin^j\theta,
\end{array}
\end{equation}
for $\l=1,2,\ldots,d$.

\smallskip

System \eqref{eq.systemX.new.coordinates} is $2\pi$--periodic 
with respect to the independent variable $\theta$, and satisfies 
the hypotheses of the Theorem \ref{Th.continuous.averaging} for 
$\varepsilon_0$ small and $D$ fixed. Now we compute  the averaged 
function $f=(f_1,f_{2}, \ldots, f_{d+1})$ given in 
\eqref{eq.averaging.function.th} with $T=2\pi$. 

\smallskip

Now taking
$$
\mu_{(p,q)} = \displaystyle\int_0^{2\pi}\cos^p\theta\sin^q \theta\der\theta,
$$
note that $ \mu_{(p,q)} \neq  0 $   if only if $p$ and $q$ are simultaneously even.

\smallskip

So we obtain
\begin{equation} 
\label{eq.f2.continuous}
\begin{array}{lll}
f_1(r,\z)&=& \dis \int_0^{2\pi} F_1(\theta,r,\z)\der \theta  = \dis \sum_{\footnotesize \begin{matrix} i+j+k=0 \\ i~odd,~j~even \end{matrix}}^{n} r^{i+j} \z^k~ a_{ijk} ~\mu_{(i+1,j)} \\
    & & +  \dis \sum_{\footnotesize \begin{matrix} i+j+k=0 \\i~even,~j~odd
\end{matrix}}^{n} r^{i+j} \z^k~ b_{ijk} ~\mu_{(i,j+1)} \vspace*{0.2cm}\\
f_{\l+1}(r,\z)&=& \dis \int_0^{2\pi} F_{\l+1}(\theta,r,\z)\der \theta =  \dis \sum_{\footnotesize\begin{matrix} i+j+k=0 \\ i,j~even
           \end{matrix}}^{n}  r^{i+j} \z^k~c_{\l ijk} ~ \mu_{(i,j)},
\end{array}
\end{equation} 
for $\l=1,2,\ldots,d$. 

\smallskip

We split the proof in two parts. First assume $n$ is odd, so
\begin{equation}
f_1(r,\z)=A_1 r + A_3 r^3 + \ldots + A_n r^n \nonumber
\end{equation}

\noindent where
\begin{equation}
\label{eq.Ap.continuous}
A_p = \sum_{k=0}^{n-p} \z^k \left( \sum_{\footnotesize \begin{matrix} i+j=p \\
i~odd,j~even \end{matrix}}  a_{ijk} ~\mu_{(i+1,j)} + \sum_{\footnotesize \begin{matrix}
i+j=p \\ i~even,j~odd \end{matrix}} b_{ijk} ~\mu_{(i,j+1)} \right),
\end{equation}
for $p=1,2,\ldots,n$. So we write $f_1=r \bar{f}_1$ with
$$
\bar{f}_1(r,\z)= A_1 + A_3 r^2 + \ldots + A_n r^{n-1}.
$$

Since $r>0$  it is sufficient to solve $(\bar{f}_1,f_{2}, \ldots,
f_{d+1})=(0,\ldots,0)$ to determine the number of solutions of $f
\equiv 0$. As $\bar{f}_1$ is a polynomial in the variables $r$ and
$\z \in \R^d$ of degree $n-1$ and  $f_{l+1}$  are polynomials in the
variables $r$ and $\z \in \R^d$ of degree  $n$ for $l=1,2,\ldots,d$,
by B\'{e}zout's theorem (see \cite{Ful}) $(\bar{f}_1,f_{2}, \ldots,
f_{d+1})$ has at most $n^d(n-1)$ solutions. However $\bar{f}_1$ is
even on variable $r$ then we consider only solutions with $r>0$, in
this case the maximum number of solutions of $f \equiv 0$ is
$n^d(n-1)/{2}$.

\smallskip

Now we  prove that this number is reached. For this, we exhibit a
particular case for which this occurs.  Let  $a_{ij0},b_{ij0} \neq
0$ and we take zero all the other $a_{ijk},b_{ijk}$, then
$\bar{f}_1(r,\z)$ is a real polynomial that does not depend of $\z\in
\R^d$ of degree $n-1$, that is
$$
\bar{f}_1= A_1 + A_3 r^2 + \ldots + A_n r^{n-1}
$$
where 
\begin{align*}
A_p=\sum_{\footnotesize \begin{matrix} i+j=p \\ i~odd,~j~even
\end{matrix}} a_{ij0}  ~ \mu_{(i+1,j)} + \sum_{\footnotesize \begin{matrix}  i+j=p \\ i~even, ~ j~odd \end{matrix}} b_{ij0}  ~ \mu_{(i,j+1)}  .
\end{align*}

On the other hand,  we take $c_{\l ijk}=0$ if $i,j, k_1,
k_{l-1},k_{l+1}, \ldots, k_d \neq 0$ for each  $\l=1,2,\ldots,d$ so that
$$
f_{\l+1}(r,z)= \sum_{k_{\l}=0}^{n}  {z_{\l}}^{k_{\l}} ~(2 \pi ~c_{{\l} 00 k_{\l}}).
$$

\smallskip

In this particular case, we can take $a_{ij0},b_{ij0}$ and
$c_{\l 00 k_{\l}}$ in order that all coefficients of $\bar{f}_1$ and
$f_{\l+1}$  are  linearly independent for all $\l=1,2,\ldots,d$. So we
can choose these coefficients in such a way that $\bar{f}_1$ has
${(n-1)}/{2}$ simple positive real roots and $f_{\l+1}$ has $n$
simple real roots for each $\l=1,2,\ldots,d$. Then $(\bar{f}_1,f_{2},
\ldots,f_{d+1})$ has $n^d(n-1)/{2}$ solutions with $r>0$.

\smallskip

Now assume $n$ is even. Then $f=r\bar{f}_1$  where
$$
\bar{f}_1(r,z)= A_1 + A_3r^2 + \ldots + A_{n-1} r^{n-2}.
$$
with $A_p$ given in \eqref{eq.Ap.continuous}.

\smallskip

Note that $\bar{f}_1$ has degree $n-1$ as a polynomial in the
variables $r$ and $z$ and $f_{\l+1}$ given in
\eqref{eq.f2.continuous} has degree $n$ as polynomials in the
variables $r$ and $z$ for all $l=1,2,\ldots,d$, so
$(\bar{f}_1,f_{2}, \ldots, f_{d+1})$ have at most $ n^d(n-1)/{2}$
solutions of type $(r,z)$ with $r>0$.

\smallskip

To prove that this number is reached we consider
$a_{10k_1},b_{01k_1}\neq 0$  and we take zero all the other
$a_{ijk},b_{ijk}$. Then
$$
\bar{f}_1(r,\z)=A_1=\sum_{k_1=0}^{n-1} z_1^{k_1} \left( a_{10k_1} ~\mu_{(2,0)}
+   b_{01k_1} ~\mu_{(0,2)}
 \right)
$$
is a complete polynomial on variable $z_1$ of degree $n-1$. Now  for $\l=2, 3\ldots,d$ we
take $c_{1ij0}, c_{\l 00k_{\l}} \neq 0$  and we
take zero all the other $c_{\l ijk}$ so that
\begin{align*}
f_{2}(r,\z) &= \sum_{\footnotesize \begin{matrix} i+j=0 \\i,j~even
\end{matrix}}^{n} r^{i+j} ~c_{1ij0}  ~\mu_{(i,j)}  ,\\
f_{\l+1}(r,\z) &=\sum_{k_{\l}=0}^{n} z_\l^{k_\l}~  ~c_{\l 00k_{\l}} 2\pi,
\end{align*}
for $\l=2, 3\ldots,d$.

\smallskip 

Here we take $a_{10k_1},~b_{01k_1}$, $c_{1ij0}$ and $c_{\l 00k_{\l}}$ in
such a way for that all coefficients of $\bar{f}_1$ and $f_{\l+1}$
for $\l=1,2,\ldots,d$, are linearly independent. Therefore we can
choose these coefficients in order that $\bar{f}_1$ has $n-1$ simple
real roots, $f_{2}$ has ${n}/{2}$ simple positive real roots and
$f_{\l+1}$ has ${n}$ simple real roots for each $\l=2,3,\ldots,d$ . In
this case $(\bar{f}_1,f_{2},\ldots, f_{d+1})=(0,\ldots,0)$ has
$n^d(n-1)/2$ solutions with $r>0$.  Furthermore by  independence of
the coefficients these solutions can be taken in a way that the
Jacobian of $f$ in all these solutions is nonzero.

\smallskip

This completes the proof of  Theorem \ref{th.cycles.systemX}.

\subsection{Proof of Theorem \ref{th.cycles.discontinuous}}
\label{section:proof.th.discontinuous}

 Analogously to what we did in the previous section we shall study the existence of limit cycles bifurcating of periodic solutions of system \eqref{eq.perturbed.discontinuous} when $\varepsilon=0$ which are contained in
the cylindrical annulus $\tilde{A}$ defined in \eqref{2*}. 

\smallskip 

In $\tilde{A}$ we have for
$\varepsilon$ sufficiently small $\dot{\theta} >0$ for all $(\theta,r,z)
\in \tilde{A}$.  Taking  $\theta$ as independent variable system
\eqref{eq.perturbed.discontinuous} in $\tilde{A}$ becomes

 \begin{equation}
 \label{S.new.coord}
 \begin{array}{c}
 \left.
 \begin{array}{lll}
 r'~ &=  & \e  F_1(\theta,r, \z) + \mathcal{O}(\e^2) , \vspace*{0.2 cm}\\
 \z'_{\l} & = & \e F_{\l+1}(\theta,r, \z) + \mathcal{O}(\e^2)
  \end{array}
 \right\} \quad \hbox{if}\quad h(\theta,r,\z) >0 ,\vspace*{0.3 cm} \\ 
 \left.
 \begin{array}{lll}
 r'~ &=  & \e  G_1(\theta,r, \z) + \mathcal{O}(\e^2) , \vspace*{0.2 cm}\\
 \z'_{\l} & = & \e G_{\l+1}(\theta,r,\z) + \mathcal{O}(\e^2)
 \end{array}
 \right\} \quad \hbox{if}\quad h(\theta,r, \z) <0 ,
 \end{array}
 \end{equation}
where
\begin{equation} 
\label{eq.G1.G2}
\begin{array}{lll}
G_1(\theta,r,\z)& = &\dis\sum_{i+j+k=0}^{n} r^{i+j} \z^k  \big(  \alpha_{ijk}
\cos^{i+1}\theta \sin^j\theta  \vspace*{0.15cm}\\
& & +  \beta_{ijk}\cos^{i}\theta \sin^{j+1}
\theta \big), \vspace*{0.15cm} \\
G_{\l+1}(\theta,r,\z)& = &\dis\sum_{i+j+k=0}^{n} r^{i+j} \z^k \gamma_{\l ijk}
\cos^i\theta \sin^j\theta,
\end{array}
\end{equation}

for  $\l=1,2,\ldots,d$, $F_1$ and $F_2$ given in \eqref{eq.F1.F2}, and  $h(\theta,r,\z)= \sin \theta$. Then for $(\theta,r,\z) \in
h^{-1}(0)$ we have $\dis\frac{\partial h}{\partial \theta}(\theta,
r,z)_{\mid {\theta \in \lbrace 0,\pi \rbrace }}=\cos\theta_{ \mid {\theta
\in \lbrace 0,\pi \rbrace}}=\pm 1 \neq 0$.

\smallskip

Now, let

$$
I_{(p,q)}=\displaystyle \int_0^\pi \cos^p\theta\sin^q\theta\der\theta \quad\text{and}
\quad J_{(p,q)} = \displaystyle \int_\pi^{2\pi} \cos^p\theta\sin^q \theta \der \theta.
$$ 

Then $(-1)^q  J_{(p,q)}= I_{(p,q)}$ and $I_{(p,q)}=J_{(p,q)}=0$ if only if $p$ is odd. Thus  we have
\begin{align}
\label{f1f2}
f_1(r,\z)= &\int_{0}^{\pi}F_1(\theta,r,\z) \der \theta + 
\int_{\pi}^{2\pi}G_1(\theta,r,\z) \der \theta \nonumber\\
= &\sum_{\footnotesize \begin{matrix} i+j+k=0 \\ i~odd \end{matrix}}^{n}
 r^{i+j} \z^k \left( a_{ijk} +(-1)^j \alpha_{ijk} \right) I_{(i+1,j)} \nonumber \\
& + \sum_{\footnotesize \begin{matrix} i+j+k=0 \\i~even
\end{matrix}}^{n} r^{i+j} \z^k \left(b_{ijk}-(-1)^j\beta_{ijk} \right)
I_{(i,j+1)},\\
f_{\l+1}(r,\z)=& \int_{0}^{\pi}F_{\l+1}(\theta,r,\z) \der\theta + \int_{\pi}^{2\pi}
G_{\l+1}(\theta,r,\z) \der\theta \nonumber \\
=& \sum_{\footnotesize\begin{matrix} i+j+k=0 \\ i~even \end{matrix}}^{n}
r^{i+j} \z^k~  (c_{\l ijk} +(-1)^j \gamma_{\l ijk}) ~  ~I_{(i,j)}, \nonumber
\end{align}
for $\l=1,2,\ldots,d$.

\smallskip

Then $f_{\l+1}$ are polynomials on variables $r$ and $z$ of degree $n$ for each $\l=0,1, \ldots, d$. By B\'{e}zout's theorem $f$ has at most  $n^{d+1}$ solutions (with $r >0$).

\smallskip

To prove that this number is reached we choose a particular e\-xam\-ple. So  take $a_{ij0} - (-1)^j \alpha_{ij0} , b_{ij0} - (-1)^j \beta_{ij 0} \neq 0$ and we
take zero all the other $a_{ij k} - (-1)^j \alpha_{ij k} , b_{ij k} - (-1)^j \beta_{ij k } $. Then ${f}_1$ is a real polynomial on $r$ of degree $n$ that does not depend of $\z$.

\smallskip

Analogously for each $\l=1,2,\ldots,d$ we take
$c_{\l 00 k_\l} + \gamma_{\l 00 k_\l} \neq 0$ and we
take zero all the other $c_{\l ij k} + \gamma_{\l ij k} $,  so that
$$
f_{\l+1}(r,z) = \sum_{k_\l=0}^{n} z_\l^{k_\l} ~(c_{\l 00 k_\l} +
\gamma_{\l 00 k_\l})\pi . 
$$

\smallskip

Under such conditions all coefficients of ${f}_1$ and $f_{\l+1}$
for $\l=1,2,\ldots,d$ can be taken to be linearly independent from
the appropriate choice of $a_{ij0},~ \alpha_{ij0},
~b_{ij0},~\beta_{ij0}$, $c_{\l 00 k_\l}$ and $\gamma_{\l 00 k_\l}$. Such values
can be taken so that ${f}_1$ is a complete real polynomial on
variable $r$ of degree $n$,  and therefore it can have $n$
simple positive real  roots. Furthermore for each $\l=1,2,\ldots,d$ the polynomials $f_{\l+1}$ is a complete real
polynomial on variable $z_\l$ with $n$ simple real roots. So for this particular case the number of zeros of $f$ is $
n^{d+1}$. By independence of the coefficients such solutions can be
taken so that the Jacobian of $f$ evaluated at them is nonzero.

\smallskip

By Theorem \ref{Th.discontinuous.averaging} this completes the proof
of Theorem \ref{th.cycles.discontinuous}.

\subsection{Proofs of Corollaries \ref{cor.1} and \ref{cor.2}}
\label{section:proof.cor}

\begin{proof}[Proof of Corollary \ref{cor.1}]
To prove Corollary \ref{cor.1} we follow the steps of the proof
presented in subsection \ref{section:proof.th.continuous}, however we
highlight some differences. We consider the change of coordinates
given in \eqref{eq.polar.coordinates} applied to system
\eqref{eq.systemX.perturbed} taking $a_{00k}=b_{00k}=0$ for all $k$
and instead of \eqref{1*} we obtain
\begin{align*}
\dot{\theta}&= 1+ \varepsilon \sum_{i+j+k=0}^{n} r^{i+j} z^k \left(a_{ijk}
\cos^i\theta \sin^{j+1}\theta  +  b_{ijk} \cos^{i+1}\theta \sin^j\theta \right),\\
\dot{r}&= \varepsilon \sum_{i+j+k=0}^{n} r^{i+j}z^k \left( a_{ijk}
\cos^{i+1}\theta \sin^j\theta  +  b_{ijk}\cos^{i}\theta \sin^{j+1}\theta\right),\\
\dot{z}_{\l}&= \varepsilon\sum_{i+j+k=0}^{n} c_{\l ijk}~r^{i+j} z^k \cos^i\theta \sin^j\theta,
\end{align*}
for $\l=1,2,\ldots,d$.

\smallskip

As $r$ does not appear in the denominator of $\dot{\theta}$, if
$\varepsilon$ is sufficiently small $\dot{\theta}>0$ for every 
$(\theta,r, \z)$ in a ball $B$ of an arbitrary given radius around the
origin of $\R^{d+2}$. Now in the ball $B$ the variable $r$ can be approximated to
the zero as we want, this cannot occur working with the cylindrical
annulus $\tilde{A}$ of subsection \ref{section:proof.th.continuous}.
From here the calculations are done analogously to section
\ref{section:proof.th.continuous}, and we obtain the same maximum
number of zeros of the averaging function for the new system
\eqref{eq.systemX.new.coordinates} with $a_{00k}=b_{00k}=0$.
\end{proof}

\begin{proof}[Proof of Corollary \ref{cor.2}]
The same above argument can be used for proving Corollary
\ref{cor.2} and we follow the steps of the proof presented in
subsection \ref{section:proof.th.discontinuous}. More precisely, we
apply the change of coordinates \eqref{eq.polar.coordinates} to
system \eqref{eq.perturbed.discontinuous} considering
$a_{00k}=\alpha_{00k}=b_{00k}=\beta_{00k}=0$ for all $k$ in order to
obtain  an expression for $\dot{\theta}$ in both systems $y>0$ and
$y<0$ with denominator that does not depend on $r$.  So for
$\varepsilon$  sufficiently small we have $\dot{\theta}>0$ for every $(r,\theta,
z)$ in the ball $B$ of the proof of Corollary \ref{cor.1}. 

\smallskip 

Since the exponent $i+j$ of the variable $r$ on the polynomial $f_1$, given in \eqref{f1f2}, is at least one because $i$ is odd, we have
\begin{equation}
f_1(r,\z)=  A_1 r + A_2 r^2 + \ldots + A_n r^n
\end{equation}
\noindent  with
\begin{equation}
\label{eq.Ap}
\begin{array}{lll}
A_p&=& \dis \sum_{k=0}^{n-p} ~ \sum_{\footnotesize \begin{matrix} i+j=p \\
i~odd \end{matrix}} z^k  \left( a_{ijk} +(-1)^j \alpha_{ijk} \right) I_{(i+1,j)} \\
& & + \dis \sum_{k=0}^{n-p}~ \sum_{\footnotesize
  \begin{matrix}  i+j=p \\ i~even \end{matrix}} z^k \left(b_{ijk} -
 (-1)^j \beta_{ijk} \right)   I_{(i,j+1)},
\end{array}
\end{equation}
for $p=1,2,\ldots,n$. 

\smallskip

Consequently we obtain $f_1=r \bar{f}_1$ where
$$\bar{f}_1(r,\z)= A_1 + A_2 r +
\ldots + A_n r^{n-1} .
$$

 As  $r >0$, to know the solutions of
$(f_1,f_{2}, \ldots, f_{d+1})=(0,\ldots,0)$ is equivalent to solve
$(\bar{f}_1,f_{2}, \ldots, f_{d+1})=(0,\ldots, 0)$. But, $\bar{f}_1$
is a polynomial on variables $r$ and $z$ of degree $n-1$ and the functions 
$f_{\l+1}$  given in \eqref{f1f2} are polynomials on variables $r$ and $z$ of degree $n$ for
each $\l=1,2, \ldots, d$. By B\'{e}zout's theorem $f$ has at most $n^d(n-1)$ solutions (with $r >0$).

\smallskip

To prove that this number is reached we choose a particular example. This step of this proof is done as in subsection
\ref{section:proof.th.discontinuous}  for obtaining the maximum
number of zeros of the averaging function for the new system
\eqref{S.new.coord} with
$a_{00k}=\alpha_{00k}=b_{00k}=\beta_{00k}=0$. Again in the ball $B$ the variable $r$ can be approximated to zero as we want.

\end{proof}

\section*{Acknowledgements}

The first author is partially supported by a MINECO/FEDER grant
MTM2008--03437, an AGAUR grant 2009SGR--0410, an ICREA Academia,
FP7--PEOPLE--2012--IRSES--316338 and 318999, and
FEDER/UNAB10--4E--378.  The second  author is partially supported by
a FAPESP--BRAZIL grant 2012/18780--0. The third author is partially
supported by a FAPESP-BRAZIL grant 2012/23591--1 and 2013/21078--8.

\end{document}